\documentclass[a4paper, 11pt, reqno]{amsart}

\usepackage{amsfonts, amsthm, amssymb, amsmath, stmaryrd}
\usepackage{mathrsfs,array}
\usepackage{eucal,fullpage,times,color,enumerate,accents}
\usepackage{bbm}
\usepackage{url}
\usepackage{tikz}
\usetikzlibrary{fadings}
\tikzfading[name=fade out,
  inner color=transparent!0, outer color=transparent!100]
\tikzfading[name=fade right,
  left color=transparent!0, right color=transparent!100]
\tikzfading[name=fade left,
  right color=transparent!0, left color=transparent!100]
\tikzfading[name=fade mid,
  left color = transparent!100, right color = transparent!100, middle color=transparent!0]
\usepackage{ifthen}
\tikzset{
  laser beam action/.style={
    line width=\pgflinewidth+1.4pt,draw opacity=.095,draw=#1,
  },
  laser beam recurs/.code 2 args={%
    \pgfmathtruncatemacro{\level}{#1-1}%
    \ifthenelse{\equal{\level}{0}}%
    {\tikzset{preaction={laser beam action=#2}}}%
    {\tikzset{preaction={laser beam action=#2,laser beam recurs={\level}{#2}}}}
  },
  laser beam/.style={preaction={laser beam recurs={20}{#1}},draw opacity=1,draw=#1},
}
\usepackage{graphics}
\usepackage{graphicx}
\usepackage{xypic}
\usepackage{bm}
\usetikzlibrary{calc}
\usepackage[font=small,labelfont=bf]{caption}

\definecolor{navy}{rgb}{0,0,.65}

\usepackage[colorlinks,backref]{hyperref}
\hypersetup{colorlinks=true,urlcolor=blue,linkcolor=navy,citecolor=navy}

\makeatletter
\def\@tocline#1#2#3#4#5#6#7{\relax
  \ifnum #1>\c@tocdepth 
  \else
    \par \addpenalty\@secpenalty\addvspace{#2}%
    \begingroup \hyphenpenalty\@M
    \@ifempty{#4}{%
      \@tempdima\csname r@tocindent\number#1\endcsname\relax
    }{%
      \@tempdima#4\relax
    }%
    \parindent\z@ \leftskip#3\relax \advance\leftskip\@tempdima\relax
    \rightskip\@pnumwidth plus4em \parfillskip-\@pnumwidth
    #5\leavevmode\hskip-\@tempdima
      \ifcase #1
       \or\or \hskip 1em \or \hskip 2em \else \hskip 3em \fi%
      #6\nobreak\relax
    \dotfill\hbox to\@pnumwidth{\@tocpagenum{#7}}\par
    \nobreak
    \endgroup
  \fi}
\makeatother

\newtheorem{theorem}{Theorem}[subsection]
\newtheorem{corollary}[theorem]{Corollary}
\newtheorem{lemma}[theorem]{Lemma}
\newtheorem{proposition}[theorem]{Proposition}
\newtheorem{definition}[theorem]{Definition}

\newtheorem{conjecture}[theorem]{Conjecture}

\newtheorem{quasi-theorem}[theorem]{Quasi-Theorem}
\newtheorem{blank remark}[theorem]{}

\newtheorem{Th}{Theorem}[]

\newtheorem{rem1}[theorem]{Remark}
\newenvironment{remark}{\begin{rem1}\em}{\end{rem1}}

\newtheorem{not1}[theorem]{Notation}
\newenvironment{notation}{\begin{not1}\em}{\end{not1}}

\newcommand{\A}{{\mathbb{A}}}

\newcommand{\PP}{\mathbb{P}}         
\newcommand{\QQ} {{\mathbb Q}}		
\newcommand{\RR} {{\mathbb R}}

\newcommand{\FF}{{\mathbb F}}

\DeclareMathOperator{\Gal}{Gal}
\DeclareMathOperator{\sSpec}{Spec_{\ \!}}

\newcommand{\lra}{{\longrightarrow}}
\DeclareMathOperator{\Def}{\overset{{}_{\text{def}}}{=}}

\newsavebox{\foobox}

\newcommand{\mono}{\!\xymatrix{{}\ar@{^{(}->}[r]&{}}\!}
\newcommand{\epi}{\!\xymatrix{{}\ar@{->>}[r]&{}}\!}

\newcommand{\iso}{\cong}

\begin{document}

\title{Correlations between primes in short intervals\\ on curves over finite fields}

\author{Efrat Bank}
\address{\newline {\bf Efrat Bank} \newline
University of Michigan Mathematics Department \newline
530 Church Street \newline
Ann Arbor, MI 48109-1043 \newline
United States}
\email{\href{mailto:ebank@umich.edu}{ebank@umich.edu}}

\author{Tyler Foster}
\address{\newline {\bf Tyler Foster} \newline
L'Institut des Hautes \'Etudes Scientifiques \newline
Le Bois-Marie, 35 route de Chartres \newline
91440 Bures-sur-Yvette \newline
France}
\email{\href{mailto:foster@ihes.fr}{foster@ihes.fr}}

\date{\today}

\pagestyle{plain}

\pagestyle{plain}

\maketitle

\begin{abstract}

\vskip -.5cm
	We prove an analogue of the Hardy-Littlewood conjecture on the asymptotic distribution of prime constellations in the setting of short intervals in function fields of smooth projective curves over finite fields.
\end{abstract}

\setcounter{tocdepth}{1}

\begin{section}{Introduction}

	In recent work \cite{BB15}, Bary-Soroker and the first author use their work with Rosenzweig \cite{BBR15} on the asymptotic distribution of primes inside short intervals in $\FF_{\!q}[t]$ to establish a natural counterpart, over the field $\FF_{\!q}(t)$, to the still unsolved Hardy-Littlewood conjecture. In \cite{BF17}, the two present authors show how to extend the results of \cite{BBR15} to give asymptotic distributions of primes in short intervals on the complement of a very ample divisor in a smooth, geometrically irreducible projective curve $C$ over a finite field $\FF_{\!q}$. In the present paper, we apply ideas from \cite{BB15} and \cite{BF17} to prove a natural counterpart to the Hardy-Littlewood conjecture on the complement of a very ample divisor in a smooth, geometrically irreducible projective curve $C$ over a finite field $\FF_{\!q}$.
	
\begin{subsection}{The Hardy-Littlewood conjectures for short intervals}
	Let $K$ be a number field, $\mathcal{O}_{K}$ its ring of integers, and $N(h)\Def\#\mathcal{O}_{K}/(h)$ the norm on elements $h\in\mathcal{O}_{K}$. Given an $n$-tuple $\mathbf{\sigma}=(\sigma_1,...,\sigma_n)$ of elements in $\mathcal{O}_{K}$, denote by $\pi_{K,\mathbf{\sigma}}(x)$ the $n$-tuple principal prime counting function
	$$
	\pi_{K,\mathbf{\sigma}}(x)
	\ \Def\ 
	\#\big\{ h\in\mathcal{O}_{K} : 2<N(h)\leq x,\text{ and each }(h+\sigma_i)\subset\mathcal{O}_{K}\text{ is prime}\big\}.
	$$
The Hardy-Littlewood $n$-tuple conjecture \cite{HL23} (henceforth the HL conjecture), asserts that:
	
\begin{conjecture}[Hardy-Littlewood]\label{conj:HL}
\normalfont
	The function $\pi_{\QQ,\sigma}$ satisfies the asymptotic formula
	$$
	\pi_{\QQ,\sigma}(x)\ \sim\ \frak{S}(\sigma)\ \!\frac{x}{(\log x)^n},\qquad x\to\infty
	$$
for a positive constant $\frak{S}(\sigma)$.
\end{conjecture}

	Despite much numerical verification and many proof attempts, the HL conjecture remains open. Results toward the conjecture include the work of Goldston, Pintz and Yildirim \cite{GPY}, Zhang \cite{Zhang} and Maynard \cite{Maynard}. See Graville \cite{Granville2010} for a partial history of the problem.

	One can also formulate short interval variants of the HL conjecture. As in \cite[Discussion before Conjecture~1.3.1]{BF17}, the ambiguity in defining subsets of $\mathcal{O}_{K}$ in terms of the norm on $K$ leads to at least two possible formulations. Fix a $n$-tuple $\sigma=(\sigma_1,...,\sigma_n)$ of elements in $\mathcal{O}_{K}$. Then each real number $1>\varepsilon>0$ determines a family of intervals $I(x,\varepsilon)\Def [x-x^{\varepsilon}, x+x^{\varepsilon}]\subset \RR$, with corresponding prime counting function
	$$
	\pi_{K,\mathbf{\sigma}}(I(x,\varepsilon))
	\ \Def\ 
	\#\big\{ N_{K}(h)\in I(x,\varepsilon):\text{each }(h+\sigma_{i})\subset\mathcal{O}_{K}\text{ is prime}\big\}.
	$$
On the other hand, if $S=\{\mbox{infinite places of }K\}$ and $\varepsilon_{S}=(\varepsilon_{\mathfrak{p}})_{\mathfrak{p}\in S}$ is an $\#S$-tuple of real numbers $1>\varepsilon_{\mathfrak{p}}>0$, then for each $b\in\mathcal{O}_{K}$, we can define
	\begin{equation*}
	I(b,\varepsilon_{S})\Def \big\{ a\in \mathcal{O}_{K} :  |a-b|_{\mathfrak{p}}\le |b|^{\varepsilon_{\mathfrak{p}}}_{\mathfrak{p}}\mbox{ for each }\mathfrak{p}\in S \big\},
	\end{equation*}
with corresponding prime counting function
	$$
	\pi_{K,\sigma}\big(I(b,\varepsilon_{S})\big)
	\ =\ 
	\#\big\{a\in I(b,\varepsilon_{S}):\mbox{ each }(a+\sigma_i)\subset\mathcal{O}_{K}\mbox{ is prime} \big\}.
	$$	
 
\begin{conjecture}\label{conj:HL short interval number fields}
\normalfont
	{\bf (i).} The function $\pi_{K,\mathbf{\sigma}}\left(I(x,\epsilon)\right)$ satisfies
	\begin{equation*}
	\pi_{K,\mathbf{\sigma}}\left(I(x,\epsilon)\right)
	\ \sim\ 
	\frak{S}(\sigma)\frac{\#I(x,\epsilon)}{(\log x)^{n}},\qquad x\to \infty,
	\end{equation*}
where $\frak{S}(\sigma)$ is some positive constant depending on $\sigma$ and on the class number of $K$.

	{\bf (ii).} The function $\pi_{K,\sigma}\big(I(b,\varepsilon_{S})\big)$ satisfies
	$$
	\\pi_{K,\sigma}\big(I(b,\varepsilon_{S})\big)
	\ \ \sim\ \ 
	\frak{S}(\sigma)\frac{\#I(b,\varepsilon_{S})}{(\log N_{\!K\!}(b))^{n}},
	\qquad N_{\!K\!}(b)\to\infty,
	$$
where $\frak{S}(\sigma)$ is some positive constant depending on $\sigma$ and on the class number of $K$.	
	\end{conjecture} 

Gross and Smith \cite{gross2000} present numeric evidence for Conjecture \ref{conj:HL short interval number fields}.(i) for certain number fields and \textit{reasonable sets}, which are regions that may be interpreted as short intervals. Based on this, they present a conjecture similar to Conjecture \ref{conj:HL short interval number fields}.(i) above.
	
	Notice that when $K=\QQ$, then $S=\{\infty\}$ and Conjectures \ref{conj:HL short interval number fields}.(i) and (ii) coincide. Natural analogues of Conjecture \ref{conj:HL} and \ref{conj:HL short interval number fields} also hold when we replace $\QQ$ with field of rational functions $\FF_{\!q}(t)$ over a finite field $\FF_{\!q}$ in the large $q$ limit. For $f\in\FF_{\!q}[t]$ monic and $\varepsilon>0$ a real number, define
	$
	I(f,\varepsilon)
	\Def
	\big\{h\in\FF_{\!q}[t] : | h-f| \leq | f |^{\varepsilon}   \big\}
	$,
with $|f|\Def q^{\deg f}$. For $n$-tuples $\mathbf{f}=(f_1,...,f_n)$ of distinct polynomials $f_{i}$ and distinct polynomials in $\FF_{\!q}[t]$, define
	\begin{equation}\label{eq: prime poly counting func}
	\pi_{q,\mathbf{f}}(k)
	\ \Def\ 
	\# \big\{ h\in \FF_{\!q}[t]: h\text{ monic of degree } k\text{, and each } h+f_{i}\text{ is a prime polynomial}\big\};
	\end{equation}
	$$
	\begin{array}{rcl}
	\pi_{q,\mathbf{f}}\big(I(f,\varepsilon)\big)
	&
	\!\!\Def\!\!
	&
	\#\big\{h\in I(f,\varepsilon):\text{ each } h+f_{i}\text{ is a prime polynomial}\big\}.
	\end{array}
	$$
	
\begin{theorem}\label{conj:HLPolyAnalogue}
	\normalfont
	Fix an integer $B>n$. Then:
	
	{\bf (i)} {(Pollack \cite{Pollack-twinPrime, PrimeSpecialization_Pollack2008}, Bary-Soroker \cite{HL_Lior2014}, Carmon \cite{Carmon_Char2}).}
	The function $\pi_{q,\bold{f}}(k)$ satisfies
	$$
	\pi_{q,\mathbf{f}}(k)\ =\ \frac{q^k}{k^n}\Big(1 + O_{B}(q^{-1/2})\Big)
	$$
uniformly in $f_1,...,f_n$ with degrees bounded by $B$ as $q\lra \infty$.
\vskip .2cm
	
	{\bf (ii)}
{(Bank \& Bary-Soroker \cite[Theorem 1.1]{BB15}).} The function $\pi_{q,\mathbf{f}}(I(f_0,\varepsilon))$ satisfies
	
	$$
	\pi_{q,\mathbf{f}}(I(f_0,\varepsilon))
	\ =\ \frac{\# I(f_0,\varepsilon)}{\underset{i=1}{\overset{n}{\mbox{{\larger\larger$\prod$}}}}\deg(f_0+f_{i})}\Big( 1 + O_{B}(q^{-1/2})\Big)
	$$
uniformly for all $f_{1},\dots,f_{n}$ of degree at most $B$, and for all monic polynomials $f_0$ satisfying $\frac{2}{\varepsilon}<\deg f_0<B$, as $q\lra \infty$ an odd prime power.
	\end{theorem}
	
\end{subsection}

	
\begin{subsection}{Main results: Correlation of primes in short intervals on curves}\ 
		Let $C$ be a smooth projective geometrically irreducible curve over a finite field $\FF_q$. Fix an effective divisor $E=m_1\mathfrak{p}_{1}+\cdots +m_s\mathfrak{p}_{s}$ on $C$. Let $\sigma\ \Def\ (\sigma_1,...,\sigma_n)$ be an $n$-tuple of rational functions $\sigma_{i}$ on $C$, regular on $C\backslash E\Def C\backslash\text{supp}(E)$. Define an $n$-tuple principal prime counting function 
	\begin{equation}\label{eq: n-tuple prime counting func on curves}
	\pi_{C,\mathbf{\sigma}}(E)
	\ \ \Def\ \ 
	\#\left\{\!\!
	\begin{array}{c}	
	h\in K(C)\mbox{ such that }\text{div}(h)_{-}=E,\mbox{ and each}
	\\
	\text{function }h+\sigma_1, \dots, h+\sigma_n\mbox{ generates  a prime}
	\\
	\text{ideal in the ring of regular functions on }  C\backslash E
	\end{array}
	\!\!\right\}.
	\end{equation}
\begin{remark}
	The condition that $\text{div}(h)_{-}=E$ should be thought of as analogous to the condition $\deg h=k$ in \eqref{eq: prime poly counting func}. In the special case where $C=\PP^{1}$, the ring of rational functions on $C$ regular away from the point $\infty\in\PP^{1}$ is the polynomial ring $\FF_{\!q}[t]$. Thus when $E=\infty$, the function $\pi_{C,\sigma}$ defined in \eqref{eq: n-tuple prime counting func on curves} reduces to the function $\pi_{q,\bold{f}}$ defined in \eqref{eq: prime poly counting func}.
\end{remark}

\begin{conjecture}\label{conj:HLhighergenusAnalogue}
	\normalfont
	The function $\pi_{C,\sigma}(E)$ satisfies
	$$
	\pi_{C,\sigma}(E)\ \sim\ \frak{S}(\sigma)\ \!\frac{q^{\text{deg}_{\ \!}E}}{(\text{deg}_{\ \!}E)^n},\qquad q \to \infty
	$$
for a positive constant $\frak{S}(\sigma)$ depending on $\sigma$, uniformly in $\sigma_1,...,\sigma_n$ with bounded degree.
\end{conjecture}

\begin{definition}
\normalfont
	Given a regular function $f$ on $C\backslash E$, the {\em interval} ({\em of size $E$ around $f$}) is the set
	$$
	\ \ \ \ \ \ \ \ \ 
	\begin{array}{rcl}
	I(f,E)
	&
	\!\!\Def\!\!
	&
	\left\{\!\!
	\begin{array}{c}
	\mbox{regular functions $h$ on }C\backslash E\ \text{such}
	\\
	\mbox{that }\nu_{\mathfrak{p}_{i}}(h-f)\ge -m_{i}\ \mbox{ for all }\ 1\le i\le s
	\end{array}
	\!\!\right\}
	\\[15pt]
	&
	\!\!=\!\!
	&f+H^{0}\big(C,\mathscr{O}(E)\big).
	\end{array}
	$$
The interval $I(f,E)$ is a \textit{short interval} if the order of the pole of $f$ at each $\mathfrak{p}_{i}$ is at least $m_i$, and strictly greater than $m_{i}$ for at least one $\mathfrak{p}_{i}$. Define
	$$
	\pi_{C,\mathbf{\sigma}}(I(f,E))
	\ \ \Def\ \ 
	\#\left\{\!\!
	\begin{array}{c}
	\text{$h\in I(f,E)$ such that $h+\sigma_1,...,h+\sigma_n$ generate prime}
	\\
	\text{ideals in the ring of regular functions on } C\backslash E
	\end{array}
	\!\!\right\}.
	$$
\end{definition}

	Our main result is an analogue of Conjecture \ref{conj:HL short interval number fields} that extends Theorem \ref{conj:HLPolyAnalogue}.(ii) to curves of arbitrary genus over $\FF_{\!q}$. Fix a smooth projective geometrically irreducible curve $C$ over $\FF_{q}$. Let $E=m_{1}\mathfrak{p}_{1}+\cdots+m_{r}\mathfrak{p}_{r}$ be an effective divisor on $C$, and let $f_0,\sigma_{1},...,\sigma_{n}$ be distinct regular functions on $C\backslash E$ satisfying $-\nu_{\mathfrak{p}}(f_0)> m_{\mathfrak{p}}$ and $\nu_{\mathfrak{p}}(f_0) \neq \nu_{\mathfrak{p}}(\sigma_{i})$ for each $1\le i\le n$.

\begin{Th}\label{theorem: main theorem}
	\normalfont
	Fix an integer $B>n$. If $\text{char}_{\ \!}\FF_{\!q}\ne2$ and $E\ge 3E_{0}$ for some effective divisor $E_{0}$ on $C$ with $\text{deg}_{\ \!}E_{0}\ge2g+1$, then the asymptotic formula
	$$
	\pi_{C,\sigma}\big(I(f_{0},E)\big)
	\ =\ 
	\frac{\#I(f_{0},E)}{\prod_{i=1}^{n}\deg (f_0+\sigma_{i})}\Big(1+O_{B}\big(q^{-1/2}\big)\Big)
	$$
holds uniformly for all $E$ and $f_{0},\sigma_{1},\dots,\sigma_{n}$ as above satisfying $\text{deg}\big(\text{div}(f_{0}+\sigma_{i})\big|_{E}\big)<B$, and as $q\lra \infty$ an odd prime power.
\end{Th}

\end{subsection}

	
\begin{subsection}{Outline of paper}
	Our strategy in proving Theorem \ref{theorem: main theorem} is similar in spirit to \cite{BB15}. In \S\ref{sec:gal group calculation} we briefly review the necessary divisor theory and show that the splitting fields of distinct linear functions, evaluated at the generic element $\mathcal{F}_{\!\bold{A}}$ of a short interval, are linearly disjoint. 
We use this to show that the Galois group of the product of several linear functions evaluated at $\mathcal{F}_{\!\bold{A}}$ is isomorphic to a direct product of symmetric groups. In \S\ref{sec: Counting argument} we use a Chebotarev-type density theorem to estimate $\pi_{C,\sigma}\big(I(f_{0},E)\big)$. In \S\ref{Proof of Theorem A}, we prove the main Theorem \ref{theorem: main theorem} as a special case of the more general Theorem~\ref{theorem: main theorem factorization type} , which deals with general factorization types.

\end{subsection}

	
\begin{subsection}{Acknowledgments}
	The authors would like to thank Lior Bary-Soroker for suggesting a version of this problem. We would like to thank Jeff Lagarias and Mike Zieve for many helpful conversations. We also extend a warm thank you to Jordan Ellenberg, Alexei Entin, and Zeev Rudnick for helpful perspectives. 
	
	 The research that lead to this paper was conducted while the first author was at the University of Michigan and while  the  second  author  was  a  visiting  researcher  at  L`Institut  des  Hautes \'Etudes Scientifiques and at L'Institut Henri Poincar\'e. The first author thanks the AMS-Simons Travel Grant for supporting her visit to IHES. The second author thanks Le Laboratoire d`Excellence CARMIN for their financial support.

\end{subsection}
	
\end{section}


\begin{section}{Galois group calculation}\label{sec:gal group calculation}

\begin{subsection}{Relevant background}\label{subsection: relevant background}
	We make use of the theory of divisors on algebraic varieties. Necessary background appears in \cite[\S2]{BF17}, with further details in \cite[§II.6]{Hartshorne}. For each point $\mathfrak{p}\in C$, let $\nu_{\mathfrak{p}}$ denote valuation at $\mathfrak{p}$, applied to both functions and divisors on $C$.
	
	For the remainder of the present \S\ref{sec:gal group calculation}, we fix an effective very ample divisor $E$ on $C$ and a function $f_0$ regular on $C\backslash E$ satisfying
	$$
	-\nu_{\mathfrak{p}}(f_0)
	\ >\ 
	\nu_{\mathfrak{p}}(E)
	\ \ \ \ \ \ \mbox{for all}\ \ \ \mathfrak{p}\in\text{supp}(E).
	$$
We require a somewhat stronger positivity condition on $E$ than ``effective and very ample." Namely, assume that there exists a divisor $E_{0}$ on $C$ such that
	\begin{equation}\label{equation: thrice very ample}
	E\ \ge\ 3E_{0}.
	\end{equation}
	
	Observe however that if $E$ is effective and very ample but fails to satisfy \eqref{equation: thrice very ample}, then we can replace $E$ by $3E$ to achieve \eqref{equation: thrice very ample}.
Let $\sigma=(\sigma_1,...,\sigma_n)$ be an $n$-tuple of distinct rational functions on $C$, each regular on $C\backslash E$, and each satisfying the inequalities
	\begin{equation}\label{equation: key condition on sigma_i}
	\nu_{\mathfrak{p}}(\sigma_{i})
	\ \neq\ 
	\nu_{\mathfrak{p}}(f_{0})
	\ \ \ \mbox{and}\ \ \ 
	-\nu_{\mathfrak{p}}(\sigma_{i})
	\ >\ 
	\nu_{\mathfrak{p}}(E),
	\ \ \ \ \ \ \mbox{for all}\ \ \ \mathfrak{p}\in\text{supp}(E).
	\end{equation}
For each $\sigma_i$, define the monic linear polynomial 
	$$
	L_i(X)\ \Def\ \sigma_{i}+X
	\ \ \ \ \ \ \text{in}\ \ \ \ \ \ 
	\FF_{\!q}(C)[X].
	$$
	
	Let $R$ denote the ring of regular functions on $C\backslash E$. Fix a basis $\{1,g_{1},...,g_{m}\}$ of $H^{0}(C,\mathscr{O}(E))$ once and for all, let $\FF_{\!q}(\bold{A})=\FF_{\!q}(A_{0}, . . . , A_{m})$ denote the field of rational functions in $m+1$ variables, and define $\mathbb{A}^{\!m+1}\Def\text{Spec}_{\ \!}\FF_{q}[\bold{A}]=\text{Spec}_{\ \!}\FF_{\!q}[A_{0},...,A_{m}]$. On the trivial family of curves $(C\backslash E)\!\times_{\FF_{\!q}}\!\mathbb{A}^{\!m+1}=\text{Spec}_{\ \!}R[\bold{A}]=\text{Spec}_{\ \!}R\otimes_{\FF_{\!q}}\FF_{\!q}[\bold{A}]$, we have the regular function
	$$
	\mathcal{F}_{\!\mathbf{A}}
	\ \Def\ 
	f_{0}+A_0+\sum_{j=1}^{m}A_j\ \!g_j.
	$$
For each $\FF_{\!q}$-rational point $\bold{a}\in\mathbb{A}^{\!m+1}$, the restriction of $\mathcal{F}_{\!\bold{A}}$ to $(C\backslash E)\!\times_{\FF_{\!q}}\!\text{Spec}_{\ \!}\kappa(\bold{a})$ defines a regular function $\mathcal{F}_{\!\bold{a}}$ on $C\backslash E$.

\vskip .2cm

	Let $K\big/\FF_{\!q}(\mathbf{A})$ be an algebraic extension. For an ideal $\mathfrak{I}$ in the Dedekind domain $K\otimes_{\FF_{\!q}(\mathbf{A})}R(\mathbf{A})$, let
	\begin{equation}\label{U.P.F.}
	\mathfrak{I}\ =\ \mathfrak{P}^{e_{1}}_{1}\cdots\mathfrak{P}^{e_{\ell}}_{\ell}
	\end{equation}
denote the prime decomposition of $\mathfrak{I}$.

\begin{definition}\label{def: splitting field of an element}
\normalfont
	{\bf (i)} If $e_{1}=\cdots=e_{\ell}=1$ in \eqref{U.P.F.}, then we define the \textit{splitting field} of $\mathfrak{I}$ {\em over} $K$, denoted $\text{split}(\mathfrak{I})$ or $\text{split}(\mathfrak{I}/K)$, to be the composite
	$$
	\text{split}(\mathfrak{I})
	\ \Def\ 
	\text{split}(\mathfrak{P}_1)\cdots \text{split}(\mathfrak{P}_{\ell}),
	$$
where for each $1\le i\le n$, $\text{split}(\mathfrak{P}_i)$ denotes the normal closure of $\kappa(\mathfrak{P}_i)$ in $\overline{\FF_{\!q}}(\bold{A})$.
	
	{\bf (ii)}
	If each extension $\kappa(\mathfrak{P}_{i})/K$ is separable, then the composite extension $\text{split}(\mathfrak{I})\big/K$ is normal, and we define the \textit{Galois group of} $\mathfrak{I}$ to be
	$$
	\text{Gal}\big(\mathfrak{I}\big/K\big)
	\ \Def\ 
	\text{Gal}\big(\text{split}(\mathfrak{I})/K\big).
	$$
\end{definition}

\end{subsection}

\begin{subsection}{Linearly disjointness of splitting fields}\label{sec:Linearly disjoint}	
	For each $1\le i\le n$, the rational function $L_{i}(\mathcal{F}_{\!\bold{A}})$ on $C_{\FF_{\!q}(\bold{A})}$ determines a morphism
	$$
	L_{i}(\mathcal{F}_{\!\bold{A}}):C_{\FF_{\!q}(\bold{A})}\longrightarrow\PP^{1}_{\FF_{\!q}(\bold{A})}
	$$
and a field extension
	\begin{equation}\label{equation: field extension from L}
	\FF_{\!q}(\bold{A})\ \hookrightarrow\ \FF_{\!q}(C)(A_{0},\dots,A_{m}).
	\end{equation}
Likewise, if we define $\FF_{\!q}(\bold{A}')\Def\FF_{\!q}(A_1,...,A_m)$, then the rational function 	
	$$
	\Psi_i
	\ \Def\ 
	\sigma_{i}+f_0+\sum_{j=1}^{m}A_j\ \!g_j
	\ =\ 
	L_i(\mathcal{F}_{\!\bold{A}})-A_0.
	$$
on $C_{\FF_{\!q}(\bold{A}')}$ provides a morphism	
	$$
	\Psi_{i}:C_{\FF_{\!q}(\bold{A}')}\longrightarrow\PP^{1}_{\FF_{\!q}(\bold{A}')}
	$$
and field extension
	\begin{equation}\label{equation: field extension from Psi}
	\FF_{\!q}(\bold{A}')(t)\ \hookrightarrow\ \FF_{\!q}(C)(A_{1},\dots,A_{m}).
	\end{equation}

\begin{lemma}\label{lemma: separability for one linear polynomial}
\normalfont
	The field extension \eqref{equation: field extension from L} is separable.
\end{lemma}
\begin{proof}
	The assumption \eqref{equation: key condition on sigma_i} implies that for each $\mathfrak{p}\in\text{supp}(E)$, we have
	\begin{equation}\label{equation: inequality for sum}
	\nu_{\mathfrak{p}}(\sigma_i +f_0)
	\ =\ 
	\text{min}\{\ \!\nu_{\mathfrak{p}}(\sigma_{i}),\ \!\nu_{\mathfrak{p}}(f_{0})\ \!\}
	\ <\ 
	-\nu_{\mathfrak{p}}(E).
	\end{equation}
This makes $I(\sigma_i+f_0,E)$ a short interval in the sense of \cite[corrected Definition 1.4.1]{BF17}. Thus by \cite[Lemmas 3.3.1 \& 3.3.3]{BF17}, the variety $V\big(L_{i}(\mathcal{F}_{\!\bold{A}})\big)\subset (C\backslash E)_{\FF_{\!q}(\bold{A})}$ consists of a single point $\mathfrak{P}_{i}$ such that the field extension
	\begin{equation}\label{equation: field extension for prime}
	\FF_{\!q}(\bold{A})\ \hookrightarrow\ \kappa(\mathfrak{P}_{i})
	\end{equation}
is separable. Under the identification
	\begin{equation}\label{equation: A0 to -t}
	\begin{array}{rcl}
	\FF_{\!q}(\bold{A}) & \!\!\xrightarrow{\ \sim\ }\!\! &\FF_{\!q}(\bold{A}')(t)
	\\[4pt]
	A_{0} & \!\!\longmapsto\!\! & -t,
	\end{array}
	\end{equation}
the extensions \eqref{equation: field extension from Psi} and \eqref{equation: field extension for prime} become isomorphic. Thus \eqref{equation: field extension from Psi} is separable.
\end{proof}
   
\begin{remark}\label{rmk: disc, different and norm}
	Let $\mathfrak{P}_{i}$ be the underlying point of $V(\mathcal{F}_{\!\bold{A}})$ as in the proof of Lemma \ref{lemma: separability for one linear polynomial}. By Lemma \ref{lemma: separability for one linear polynomial}, we can consider the splitting field
	$$
	\text{split}\big({L_i(\mathcal{F}_{\!\bold{A}})}\big)
	\ \Def\ 
	\text{split}\big(\kappa(\mathfrak{P}_{i})\big/\FF_{\!q}(\bold{A})\big)
	$$
obtained as the normal closure of the extension \eqref{equation: field extension from L}, as defined in \cite[Definition 3.2.1]{BF17}. Let $\mathfrak{d}_i$ denote the discriminant of the extension $\text{split}\big({L_i(\mathcal{F}_{\!\bold{A}})}\big)\big/\FF_{\!q}(\bold{A})$, as defined in \cite[\S III.3, Definition (2.8)]{Neukirch} for instance. The discriminant is a fractional ideal in $\FF_{\!q}(\bold{A})$ that restricts to an actual ideal in $\FF_{\!q}(\bold{A}')[A_{0}]$. Because $\FF_{\!q}(\bold{A}')[A_{0}]$ is a principal ideal domain, there exists a function $D_{i}\in \FF_{\!q}(\bold{A}')[A_0]$, well defined up to multiplication by elements in $\FF_{\!q}(\bold{A}')^{\times}$, such that
	$$
	\mathfrak{d}_{i}\ =\ (D_{i})
	\ \ \ \mbox{in}\ \ \ 
	\FF_{\!q}(\bold{A}')[A_0].
	$$
	
	Via the identification \eqref{equation: A0 to -t}, \cite[\S III.3, Corollary (2.12)]{Neukirch} implies that the map $\Psi_{i}$ is ramified at a point $x$ in $C_{\FF_{\!q}(\bold{A}')}\backslash E_{\FF_{\!q}(\bold{A}')}$ if and only if $D_{i}$ vanishes at the point $\Psi_{i}(x)$ inside $\text{Spec}_{\ \!}\FF_{\!q}(\bold{A}')[A_0]\subset\PP^{1}_{\FF_{\!q}(\bold{A}')}$.

\begin{notation}
	Let $\Omega^{1}_{C}$ denote the sheaf of K\"ahler differentials on $C$. For each $1\le i\le n$, the functions $\Psi_{i}$ and $L_{i}(\mathcal{F}_{\!\bold{A}})$ are regular on $(C\backslash E)_{\FF_{\!q}(\bold{A}')}$ and $(C\backslash E)_{\FF_{\!q}(\bold{A})}$, respectively. Let $d\Psi_{i}$ and $dL_{i}(\mathcal{F}_{\!\bold{A}})$ denote their differentials, and note that
	$$
	d\Psi_{i}\ =\ dL_{i}(\mathcal{F}_{\!\bold{A}})
	\ \ \ \ \ \ \mbox{on}\ \ \ \ C_{\FF_{\!q}(\bold{A})}\backslash E_{\FF_{\!q}(\bold{A})}.
	$$
	
	Given a field extension $K/\FF_{\!q}$, a section $\omega\in\Gamma(C_{K}\backslash E_{K},\ \!\Omega^{1}_{C_{K}})$, and any point $x\in C_{K}\backslash E_{K}$, let $\omega|_{x}$ denote the restriction of $\omega$ to the fiber $(\Omega^{1}_{C_{K}})_{x}$.
\end{notation}
	
\end{remark}

\begin{lemma}\label{lem:disc prime iff sys of equations}
\normalfont
		In the notation of Remark \ref{rmk: disc, different and norm}, $D_1$ and $D_2$ are relatively prime in the polynomial ring $\FF_{\!q}(\bold{A}')[A_0]$ if and only if the system of equations
	 \begin{equation}\label{equation: separating critical system}
	\begin{array}{rcll}
	d\Psi_{1}|_{x}
	&
	\!\!\!\!=\!\!\!\!
	&
	0;
	\\[10pt]
	d\Psi_{2}|_{y}
	&
	\!\!\!\!=\!\!\!\!
	&
	0;
	\\[10pt]
	\Psi_{1}(x)
	&
	\!\!\!\!=\!\!\!\!
	&
	\Psi_{2}(y),
	\end{array}
	\end{equation}
has no solution in pairs of (not necessarily distinct) points $x,y\in(C\backslash E)_{\FF_{\!q}(\bold{A}')}$.
\end{lemma}
\begin{proof}
	Because $\Psi_{i}^{-1}(\infty)=\text{supp}(E_{\FF_{\!q}(\bold{A}')})$, Remark \ref{rmk: disc, different and norm} implies that $D_{1}$ and $D_{2}$ are relatively prime if and only if the branch divisors of $\Psi_{1}$ and $\Psi_{2}$ are disjoint in $\text{Spec}_{\ \!}\FF_{\!q}(\bold{A}')[A_0]$. Recall that the support of the branch divisor of $\Psi_{i}$ is the set of all $\mathfrak{p}\in\PP^{1}_{\FF_{\!q}(\bold{A}')}$ such that
	$$
	\mathfrak{p}\ =\ \Psi_{i}(x)
	\ \ \ \ \ \ \mbox{and}\ \ \ \ \ \ 
	d\Psi_{i}|_{x}\ =\ 0
	$$
for some $x\in C_{\FF_{\!q}(\bold{A}')}$. Thus a point $\mathfrak{p}$ of $\text{Spec}_{\ \!}\FF_{\!q}(\bold{A}')[A_0]$ in the support of both branch divisors $\Psi_{1}$ and $\Psi_{2}$ is any point satisfying $\Psi_{1}(x)=\mathfrak{p}=\Psi_{2}(y)$ and $d\Psi_{1}|_{x}=0=d\Psi_{2}|_{y}$ for some (not necessarily distinct) pair of points $x,y\in C_{\FF_{\!q}(\bold{A}')}\backslash E_{\FF_{\!q}(\bold{A}')}$.
\end{proof}

\begin{proposition}\label{proposition: no solutions exist}
\normalfont
	If $\text{char}_{\ \!}\FF_{\!q}\ne 2$, then the system of equations \eqref{equation: separating critical system} has no solution in points $x,y\in C_{\FF_{\!q}(\bold{A}')}\backslash E_{\FF_{\!q}(\bold{A}')}$. 
\end{proposition}
\begin{proof}
	Since $\Psi_{i}=L_{i}(\mathcal{F}_{\!\bold{A}})-A_{0}$, it suffices to prove that the system of equations
	 \begin{equation}\label{equation: new separating critical system}
	\begin{array}{rcll}
	dL_{1}(\mathcal{F}_{\!\bold{A}})|_{x}
	&
	\!\!\!\!=\!\!\!\!
	&
	0
	\\[10pt]
	dL_{2}(\mathcal{F}_{\!\bold{A}})|_{y}
	&
	\!\!\!\!=\!\!\!\!
	&
	0
	\\[10pt]
	L_{1}(\mathcal{F}_{\!\bold{A}})(x)
	&
	\!\!\!\!=\!\!\!\!
	&
	L_{2}(\mathcal{F}_{\!\bold{A}})(y)
	\end{array}
	\end{equation}
has no solutions over $\overline{\FF_{\!q}(\bold{A})}$.
	
	As in the proof of \cite[Proposition 4.3.3]{BF17}, choose an effective very ample divisor $E_{0}$ satisfying \eqref{equation: thrice very ample}, define $m_{0}\overset{{}_{\text{def}}}{=}\text{dim}_{\ \!}H^{0}\big(C,\mathscr{O}(E_{0})\big)-1$, and let $C\subset\PP^{m_{0}}$ be the closed embedding determined by $E_{0}$.

	Assume first that $x\neq y$. For each pair of distinct points $\xi,\eta\in C_{\overline{\FF_{\!q}}}\backslash E_{\overline{\FF_{\!q}}}$, let $t$ be a linear form on $\PP^{m_{0}}_{\overline{\FF_{\!q}}}$ such that $t(\xi)\ne t(\eta)$. The assumption on $E_{0}$ gives us a new $\overline{\FF_{\!q}}$-linear basis
	\begin{equation}\label{eq: t basis}
	\big\{1,t,t^{2},\dots,t^{\ell},g_{\ell+1},\dots,g_{m}\big\}
	\ \ \ \mbox{of}\ \ \ 
	H^{0}\big(C_{\overline{\FF_{\!q}}},\mathscr{O}(E_{\overline{\FF_{\!q}}})\big),
	\end{equation}
and a new coordinate system $A_{0},B_{1},\dots,B_{m}$ on $\mathbb{A}^{\!m+1}_{\overline{\FF_{\!q}}}$ such that
	\begin{equation}\label{eq: Li in t}
	L_{i}(\mathcal{F}_{\!\bold{A}})
	\ =\ 
	\sigma_{i}+\underset{\mbox{$=\mathcal{F}_{\!\bold{A}}$}}{\underbrace{f_{0}+A_{0}+B_{1}t+\cdots+B_{\ell}t^{\ell}+B_{\ell+1}g_{\ell+1}+\cdots+B_{m}g_{m}}}
	\end{equation}
for each $i=1,2$. For $i=1,2$, define $\Phi_{i}\Def L_{i}(\mathcal{F}_{\!\bold{A}})-B_{1}t-B_{2}t^{2}$.
	As in the proof of \cite[Proposition 4.3.3]{BF17}, we can choose a Zariski opens $U_{\xi\eta}$ that provide a covering of
	$$
	\Big((C_{\overline{\FF_{\!q}}}\backslash E_{\overline{\FF_{\!q}}})\underset{\overline{\FF_{\!q}}}{\times}(C_{\overline{\FF_{\!q}}}\backslash E_{\overline{\FF_{\!q}}})\Big)\big\backslash\{\text{diagonal}\},
	$$
such that an $\overline{\FF_{\!q}(\bold{A})}$-valued solution $(u,v)\in U_{\xi\eta}$ to \eqref{equation: new separating critical system} is the same thing as a solution to the single equation
	$$
	\text{det}(u,v)
	\ \ \Def\ \ 
	\text{det}
	\left(
	\begin{array}{ccc}
	1 & 2t(u) & \varphi_{1}(u) \\[4pt]
	1 & 2t(v) & \varphi_{2}(v) \\[4pt]
	t(v)-t(u) & t(v)^{2}-t(u)^{2} & c(u,v)
	\end{array}
	\right)
	\ \ =\ \ 
	0,
	$$
where $\varphi_{i}\Def-\frac{d\Phi_{i}}{dt}$ for $i=1,2$, and where $c(u,v)\Def\Phi_{1}(u)-\Phi_{2}(v)$. A direct calculation gives
	$$
	\text{det}(u,v)
	\ =\ 
	\big(t(v)-t(u)\big)\ \!\Big(2\ \!c(u,v)+\big(t(u)-t(v)\big)\big(\varphi_{1}(u)+\varphi_{2}(v)\big)\Big).
	$$
By the same reasoning as in \cite[Proposition 4.3.3]{BF17}, it suffices to prove that $\text{det}(u,v)$ cannot be $0$ identically on $U_{\xi\eta}$.
	
	If $n \ge 3$, then the coefficient of $B_{3}$ in 
	$$
		2\ c(u,v)+\big(t(u)-t(v)\big)\big(\varphi_{1}(u)+\varphi_{2}(v)\big)
	$$
 is $2\big(t(u)^{3}-t(v)^{3}\big)+\big(t(u)-t(v)\big)\ \!3\big(t(u)^{2}+t(v)^{2}\big)$, which is not identically $0$.
Since $t(u)\ne t(v)$, this implies that $\text{det}(u,v)$ is not identically zero on $U_{\xi\eta}$.

	Assume next that $x=y$. If the pair $x=y\in C_{\overline{\FF_{\!q}(\bold{A})}}\backslash E_{\overline{\FF_{\!q}(\bold{A})}}$ is a solution to \eqref{equation: new separating critical system}, then $L_{1}(\mathcal{F}_{\!\bold{A}})(x)=
L_{2}(\mathcal{F}_{\!\bold{A}})(x)$, implying that $(\sigma_1-\sigma_2)(x)=0$. Because $\sigma_1$ and $\sigma_2$ are distinct regular functions on $C_{\FF_q}\backslash E_{\FF_q}$, this implies that $x\in C_{\overline{\FF_q}}\backslash E_{\overline{\FF_q}}$. 
This allows us to choose a linear form $t$ on $\PP^{m_{0}}_{\overline{\FF_{\!q}}}$ such that $t(x)\neq 0$. Using this linear form $t$, choose an $\overline{\FF_{\!q}}$-linear basis \eqref{eq: t basis}. This choice lets us write $L_{i}(\mathcal{F}_{\!\bold{A}})$ as in \eqref{eq: Li in t}. The solution $x$ satisfies $dL_{1}(\mathcal{F}_{\!\bold{A}})|_{x}=0$. But our choice of $t$ lets us write
	$$
	0
	\ =\ 
	dL_{1}(\mathcal{F}_{\!\bold{A}})|_{x}
	\ =\ 
	\frac{d\sigma_{1}}{dt}(x)+\frac{df_{0}}{dt}(x)+B_1+\sum_{j=2}^{\ell}j B_{j}\ t^{j -1}(x) + \sum_{j=\ell+1}^{m}\!\!B_j\frac{dg_{j}}{dt}(x)
	$$
and so
	$$
	-B_1
	\ =\ 
	\frac{d\sigma_{1}}{dt}(x)+\frac{df_{0}}{dt}(x)+\sum_{j=2}^{\ell}j B_{j}\ t^{j -1}(x) + \sum_{j=\ell+1}^{m}\!\!B_j\frac{dg_{j}}{dt}(x).
	$$
However, the right hand side of this last equation does not involve $B_1$, contradicting the fact that $x\in C_{\overline{\FF_q}}\backslash E_{\overline{\FF_q}}$.
\end{proof}
	
\begin{corollary}\label{cor: linearly disjoint} 
\normalfont
	If $\text{char}_{\ \!}\FF_{\!q}\ne 2$, then the splitting fields $\text{split}(L_i(\mathcal{F}_{\mathbf{A}}))$, for $1\le i\le n$, are linearly disjoint over $\FF_{\!q}(\bold{A})$.
\end{corollary}
\begin{proof}
	As noted in the proof of Lemma \ref{lemma: separability for one linear polynomial}, for each $\mathfrak{p} \in \text{supp}(E)$ and for each $1\le i\le n$, the conditions \eqref{equation: key condition on sigma_i} imply that the inequality \eqref{equation: inequality for sum} holds. Thus the hypotheses of \cite[Proposition 4.3.3 and Theorem~4.1.1]{BF17} are satisfied and we have $\Gal\big(L_i(\mathcal{F}_{\mathbf{A}}),\overline{\FF_{\!q}}(\bold{A})\big)\iso S_{k_i}$ with
	$$
	k_{i}
	\ \Def\ 
	\deg_{\ \!}\text{div}(f_0+\sigma_{i})|_{C\backslash E}
	\ =\ 
	\deg_{\ \!}\text{div}\big(L_{i}(f_0)\big)\big|_{C\backslash E}.
	$$
Because we have inclusions $\Gal\big(L_i(\mathcal{F}_{\mathbf{A}}),\overline{\FF_{\!q}}(\bold{A})\big)\subseteq\Gal\big(L_i(\mathcal{F}_{\mathbf{A}}),\FF_{\!q}(\bold{A})\big)\hookrightarrow S_{k_{i}}$, this implies
	$$
	\Gal\big(L_i(\mathcal{F}_{\mathbf{A}}),\FF_{\!q}(\bold{A})\big)
	\ \iso\ 
	S_{k_i}.
	$$ 

	Let $A_{k_{i}}\!\subset\!S_{k_{i}}$ denote the alternating group on $k_{i}$ letters. Its fixed field is the quadratic extension
	$$
	\text{split}\big(L_i(\mathcal{F}_{\mathbf{A}})\big){}^{\!A_{k_{i}}}
	\ \cong\ 
	\FF_{\!q}(\bold{A})\big(\sqrt{D_{i}}\big),
	$$
where $D_{i}\in\FF_{\!q}(\bold{A}')[A_{0}]$ is an element generating the discriminant ideal $\mathfrak{d}_{i}$ as in Remark \ref{rmk: disc, different and norm}.
By \cite[Lemma 3.3]{BarySoroker2012}, it suffices to prove that the fields $\FF_{\!q}(\bold{A})\big(\sqrt{D_{i}}\big)$, for $1\le i\le n$, are linearly disjoint.

	Without loss of generality, consider $i=1,2$. By Lemma \ref{lem:disc prime iff sys of equations} and Proposition \ref{proposition: no solutions exist}, the elements $D_{1},D_{2}\in \FF_{\!q}(\bold{A}')[A_0]$ have no common prime factors. Likewise, \cite[Proposition 4.3.3]{BF17} implies that both $D_{1}$ and $D_{2}$ are square free. Because the fields $\FF_{\!q}(\bold{A})\big(\sqrt{D_{1}}\big)$ and $\FF_{\!q}(\bold{A})\big(\sqrt{D_{2}}\big)$ are degree-$2$ extensions of $\FF_{\!q}(\bold{A})$, their intersection $\FF_{\!q}(\bold{A})\big(\sqrt{D_{1}}\big)\cap\FF_{\!q}(\bold{A})\big(\sqrt{D_{2}}\big)$ is either $\FF_{\!q}(\bold{A})$ itself, or else the two field extensions coincide: $\FF_{\!q}(\bold{A})\big(\sqrt{D_{1}}\big)=\FF_{\!q}(\bold{A})\big(\sqrt{D_{2}}\big)$. The latter is the case if and only if the product $D_{1}D_{2}\in \FF_{\!q}(\bold{A}')[A_{0}]$ contains the square of a prime factor, contradicting the fact that $D_{1}$ and $D_{2}$ are square free and relatively prime in $\FF_{\!q}(\bold{A}')[A_{0}]$.
\end{proof}

Combining the results above, we have the following:
\begin{corollary}\label{prop:GalGroupOfMult} \normalfont
	For each algebraic extension $K/\FF_{\!q}$, we have natural group isomorphisms
\begin{center}
\hfill
	$
	\text{Gal}\Big(\ \!\overset{n}{\underset{i=1}{\mbox{{\larger\larger $\prod$}}}}L_{i}(\mathcal{F}_{\!\mathbf{A}}),\ \!K(\bold{A})\ \!\Big)
	\ \ \cong\ \ 
	\overset{n}{\underset{i=1}{\mbox{{\larger\larger $\prod$}}}}\ \!\text{Gal}\big(\ \!L_{i}(\mathcal{F}_{\!\mathbf{A}}),\ \!K(\bold{A})\ \!\big)
	\ \ \cong\ \ 
	S_{k_1}\!\times\cdots\times S_{k_n}.$
\hfill
$\square$
\end{center}
\end{corollary}
	
\end{subsection}



\begin{section}{Counting argument and proof of Theorem \ref{theorem: main theorem}}\label{sec: Counting argument}
	In this section we prove Theorem \ref{theorem: main theorem} the counting Proposition~\ref{prop: counting argument}. The latter provides an asymptotic formula for the number of $\FF_{\!q}$-valued points $\mathbf{a}\in\A^{\!m+1}(\FF_{\!q})$ for which each of the elements $L_{1}(\mathcal{F}_{\!\mathbf{a}}),...,L_{n}(\mathcal{F}_{\!\mathbf{a}})$ generates a prime ideal in $R$. The formulation of Proposition~\ref{prop: counting argument} should be viewed as an explicit Chebotarev theorem. Our proof makes use of \cite[Appendix A and Theorem~3.1]{ABR} and \cite[Proposition~5.1.4]{BF17}.
\end{section}

\begin{subsection}{Factorization types and general counting argument}\label{subsec: factorization type}
		Suppose given an $\FF_{\!q}$-rational point $\bold{a}\in\A^{\!m+1}(\FF_{\!q})$. If $R\big/\big(L_i(\mathcal{F}_{\!\mathbf{a}})\big)$ is a separable $\FF_{\!q}$-algebra, then because $R$ is a Dedekind domain, the ideal $\big(L_i(\mathcal{F}_{\!\mathbf{a}})\big)\subset R$ admits a prime factorization
		\begin{equation*}
		\big(L_i(\mathcal{F}_{\!\mathbf{a}})\big)
		\ =\ 
		\mathfrak{f}_{i1}\cdots\ \!\mathfrak{f}_{i\ell_{i}},
		\end{equation*} 
such that each residue field $\kappa(\mathfrak{f}_{ij})=R/(\mathfrak{f}_{ij})$ is a separable extension of $\FF_{\!q}$. The conditions on $f_0$ and $\sigma_{i}$ guarantee that in this case
		\begin{equation}\label{equation: partition for factorization type}
		k_{i}
		\ =\
		\deg_{\ \!}\text{div}\big(L_i(\mathcal{F}_{\!\mathbf{a}})\big)\big|_{C\backslash E}
		\ =\ 
		\deg_{\ \!}\text{div}(\mathfrak{f}_{i1})|_{C\backslash E}+\cdots +\deg_{\ \!}\text{div}(\mathfrak{f}_{i\ell_{i}})|_{C\backslash E}.
		\end{equation}
	
\begin{definition}\label{definition: factorization type}\normalfont
	Given an $\FF_{\!q}$-rational point $\bold{a}\in\A^{\!m+1}(\FF_{\!q})$, if $R\big/\big(L_i(\mathcal{F}_{\!\mathbf{a}})\big)$ is a separable $\FF_{\!q}$-algebra, the \textit{factorization type} $\lambda_{i, \mathbf{a}}$ of $L_i(\mathcal{F}_{\!\mathbf{a}})$ is the partition of $k_{i}$ given in \eqref{equation: partition for factorization type}.
	
	The {\em $n$-tuple factorization type counting function} of $\mathcal{L}=(L_{1},...,L_{n})$ for a fixed $n$-tuple $\lambda=(\lambda_{1},...,\lambda_{n})$, where each $\lambda_{i}$ is a partition of $k_{i}=\deg L_{i}(f_0)$, is the assignment $\pi_{C,\mathcal{L}}(-;\lambda)$ taking the short interval $I(f_0,E)$ to the value
	  $$
	  \pi_{C,\mathcal{L}}\big(I(f_{0},E);\lambda\big)
	  \ \ \Def\ \ 
	  \#
	  \big\{
	  \mathbf{a}\in \A^{\!m+1}(\FF_{\!q}):R\big/\big(L_i(\mathcal{F}_{\!\mathbf{a}})\big)\mbox{\ is separable \ and\ }\lambda_{i,\bold{a}}=\lambda_{i} \mbox{\ for\ } 1\le i\le n
	  \big\}.
	  $$

\begin{definition}\label{definition: probability}\normalfont
	Given a positive integer $N$ and a permutation $\tau\in S_{N}$, the {\em partition type} of $\tau$, denoted $\lambda_{\tau}$, is the partition of $N$ determined by the cycle decomposition of $\tau$. Having fixed a subgroup $G\subseteq S_{N}$, for each partition $\lambda$ of $N$ we define
	$$
	P(\lambda)
	\ \overset{{}_{\text{def}}}{=}\ 
	\frac{\# \{\tau\in G\lvert \lambda_{\tau}=\lambda\}}{|G|}.
	$$
	In other words, $P(\lambda)$ is the probability that a given permutation in $G$ has partition type $\lambda$.
\end{definition}
\end{definition}

\begin{remark}	
	Given two positive integers $N_{1}$ and $N_{2}$, subgroups $G_{1}\subseteq S_{N_{1}}$ and $G_{2}\subseteq S_{N_{2}}$, and partitions $\lambda_{1}$ of $N_{1}$ and $\lambda_{2}$ of $N_{2}$, we write $P(\lambda_{1})$ and $P(\lambda_{2})$ for the respective probabilities, without explicit reference to the groups $G_{1}$ and $G_{2}$.
\end{remark}

\begin{proposition}\label{prop: counting argument}\normalfont\
	Define
	$$
	\mathcal{L}(\mathcal{F}_{\mathbf{A}})
	\ \Def\ 
	L_1(\mathcal{F}_{\mathbf{A}})\cdots L_n(\mathcal{F}_{\mathbf{A}})
	\ \ \ \in\ \ \ 
	R[\mathbf{A}].
	$$
Let $G=\Gal(\mathcal{L}(\mathcal{F}_{\mathbf{A}}), \FF_q(\mathbf{A}))$, and fix a partition $\lambda=(\lambda_1,...,\lambda_n)$ of
	$$
	N
	\ \Def\ 
	\deg_{\ \!}\text{div}\big(L_{1}(f_0)\cdots L_{n}(f_{0})\big)\big|_{C\backslash E}.
	$$
Then under the assumptions in \S\ref{subsection: relevant background}, there exists a constant $c=c(B)$ such that
	$$
	\Big| \pi_{C,\mathcal{L}(\mathcal{F}_{\mathbf{A}})}\big((I(f_0,E);\lambda\big)    - P(\lambda)q^{m+1} \Big|
	\ \leq\ 
	c\ \!q^{m+1/2}.
	$$
\end{proposition}

\begin{remark}\label{rem: reducing count to polys}
	Consider the $\FF_q$-scheme $V(\mathcal{L}(\mathcal{F}_{\mathbf{A}}))$. By \cite[Remark 5.1.5]{BF17}, for each $1\le i\le n$ there exist an affine open $U_i=\sSpec \mathscr{A}_i \subset \mathbb{A}^{m+1}$ and a monic separable polynomial $\mathcal{G}_i(t)\in \mathscr{A}_{i}[t]$ such that 
	$$
	V\big(L_i(\mathcal{F}_{\mathbf{A}})\big)_{U_{i}}
	\ \cong\ 
	\sSpec \mathscr{A}_i[t]\big/\big(\mathcal{G}_{i}(t)\big).
	$$
Hence there exists an affine subscheme $U=\text{Spec}_{\ \!}\mathscr{A}\subset\bigcap_{i=1}^{n} U_i$ such that$$
	V\big(\mathcal{G}(t)\big)_{U}
	\ \iso\ 
	V(\mathcal{L}(\mathcal{F}_{\mathbf{A}}))_{U},
	\ \ \ \ \ \ \mbox{for}\ \ \ \ \ \ 
	\mathcal{G}(t)
	\ \Def\ 
	\mathcal{G}_{1}(t)\cdots \mathcal{G}_{n}(t).
	$$
\end{remark}

\begin{proof}[{\it Proof of Proposition \ref{prop: counting argument}}.]
		Let $U=\sSpec \mathscr{A}$ and $\mathcal{G}(t)\in \mathscr{A}[t]$ as in Remark~\ref{rem: reducing count to polys}. Then there exists some element $a\in\FF_q(\mathbf{A})^{\times}$ such that $a\mathcal{G}(t)\in\FF_q[\mathbf{A}][t]$. Since $V(\mathcal{G}(t))\iso V(\mathcal{L}(\mathcal{F}_{\mathbf{A}}))$
		over $U$, we have that $\text{Gal}\big(a\mathcal{G}(t),\FF_{\!q}(\mathbf{A})\big)\iso G$. Thus by \cite[Theorem~3.1]{ABR}, Proposition~\ref{prop: counting argument} holds with $a\mathcal{G}(t)$ in place of $\mathcal{L}(\mathcal{F}_{\mathbf{A}})$ for some constant $c_1(B)$. Interpreting the closed complement $Z\Def\mathbb{A}^{m+1}\backslash U$ as an $m$-cycle in $\mathbb{A}^{m+1}$, \cite[Lemma 1]{LangWeil1954} implies that there exists some constant $c_2(B)$ such that $\#Z(\FF_q)\leq c_2(B)q^{m}$. Finally, 
		$$
		\begin{array}{rcl}
		\Big| \pi_{C,\mathcal{L}(\mathcal{F}_{\mathbf{A}})}\big(I(f_0,E);\lambda\big)-P(\lambda)\ \!q^{m+1}\Big|
		&
		\leq
		&
		c_{1}(B)\ q^{m+\frac{1}{2}}+c_{2}(B)\ q^{m}
		\\[10pt]
		&
		\leq
		&
		c(B)\ q^{m+\frac{1}{2}},
		\end{array}
		$$
		where $c(B)=c_1(B)+c_2(B)$.
	\end{proof}
\end{subsection}


\begin{subsection}{Proof of Theorem \ref{theorem: main theorem}}\label{Proof of Theorem A}
	
	We obtain Theorem~\ref{theorem: main theorem} as a specific case of the following more general Theorem \ref{theorem: main theorem factorization type}, which deals with arbitrary factorization types.

\begin{Th}\label{theorem: main theorem factorization type}\normalfont
	In the conditions and notation of Theorem~\ref{theorem: main theorem}, let $\lambda$ be a partition as in Proposition~\ref{prop: counting argument}. Then
	$$
		\pi_{C,\mathcal{L}(\mathcal{F}_{\mathbf{A}})}\Big( I(f_0,E);\lambda \Big) = P(\lambda_1)\cdots P(\lambda_n)\ \!\#I(f_0,E)\ \!\Big(1 + O_{B}(q^{-1/2})   \Big).
	$$
\end{Th}
\begin{proof}[{\it Proof of Theorem \ref{theorem: main theorem factorization type}}]
	By Proposition~\ref{prop:GalGroupOfMult}, $\Gal(\mathcal{L}(\mathcal{F}_{\mathbf{A}}),\FF_q(\mathbf{A}))\iso S_{k_1}\!\times\cdots\times S_{k_n}$. Since $P(\lambda)=P(\lambda_1)\cdots P(\lambda_n)$ and $\#I(f_0,E)=q^{m+1}$, Proposition~\ref{prop: counting argument} gives 
		$$
		\begin{array}{rcl}
		\pi_{C,\mathcal{L}(\mathcal{F}_{\mathbf{A}})}\Big( I(f_0,E);\lambda \Big)
		&
		\!\!=\!\!
		&
		P(\lambda_1)\cdots P(\lambda_n)\ \!q^{m+1} +O_{B}(q^{m+1/2})
		\\[8pt]
		&
		\!\!=\!\!
		&
		P(\lambda_1)\cdots P(\lambda_n)\ \!\#I(f_0,E)\ \!\Big(1 + O_{B}(q^{-1/2})   \Big),
		\end{array}
		$$
as desired.
\end{proof}

\begin{proof}[{\it Proof of Theorem \ref{theorem: main theorem}}]
	In Theorem~\ref{theorem: main theorem factorization type}, take each $\lambda_i$ to be the partition of $k_i$ into a single cell. Then, $P(\lambda_i)=\frac{1}{k_i}=\frac{1}{\deg(f_0+\sigma_i)}$ and so
	$$
	\pi_{C,\mathcal{L}(\mathcal{F}_{\mathbf{A}})}\big( I(f_0,E);\lambda \big)
	\ =\ 
	P(\lambda_1)\cdots P(\lambda_n)\ \!q^{m+1} +O_{B}(q^{m+1/2})
	$$
	$$
	\ \ \ \ \ \ \ \ \ \ \ \ \ \ \ \ \ \ \ \ \ \ \ \ \ \ \ \ \ \ \ \ \ \ \ \ \ \ 
	\ =\ 
	\frac{\#I(f_0,E)}{\prod_{i=1}^{n}\deg_{\ \!}\text{div}(f_0+\sigma_i)|_{C\backslash E}}\ \!\Big(1 + O_{B}(q^{-1/2})   \Big),
	$$
	as desired.
\end{proof}

\begin{remark}
	In the formulation of the Hardy-Littlewood Conjecture \ref{conj:HL}, one can consider polynomial functions more general than the monic linear functions $L_{i}(X)=\sigma_{i}+X\in\mathcal{O}_{K}[X]$. Some of the resulting variant forms of Conjecture \ref{conj:HL} are well known conjectures or results in their own right. For example, Conjecture \ref{conj:HL} becomes the quantitative Goldbach conjecture if we set $n=2$ with $L_1(X)=X$ and $L_2(X)=\sigma -X$, for $\sigma\in\mathcal{O}_{K}$. If one takes $L_1(X)= X$ and $L_2(X)=\sigma_{0} + \sigma_{1} X$ for $\sigma_{0},\sigma_{1}\in\mathcal{O}_{K}$, then Conjecture \ref{conj:HL} provides an asymptotic count of primes in an arithmetic progression.
	
	 Over $\FF_{\!q}(t)$, Bary-Soroker and the first author establish a version of Theorem \ref{theorem: main theorem} for non-monic linear functions $\sigma_{i1}X+\sigma_{i0}\in\FF_{\!q}[t][X]$. We expect that this and other interesting and important variants of Theorem \ref{theorem: main theorem} hold on curves of higher genus over $\FF_{\!q}$. We plan to investigate these questions in future work.
\end{remark}

\end{subsection}

\end{section}


\vskip 1cm

\bibliography{refrencesEfrat}
\bibliographystyle{plain}

\end{document}